\documentclass[10pt,a4paper]{amsart}
\usepackage{amsthm}
\usepackage{amsmath}
\usepackage{amssymb}
\usepackage{amsfonts}
\usepackage{graphicx}
\usepackage{enumerate}

\title{Mekler's construction and tree properties}
\author{JinHoo Ahn}
\address{Department of Mathematics Yonsei University 50 Yonsei-Ro Seodaemun-Gu, Seoul 03722, South Korea}
\email{jinu1229@yonsei.ac.kr}
\thanks{The author was supported by Samsung Science Technology Foundation under Project Number SSTF-BA1301-03, and has been supported by an NRF of Korea grant 2018R1D1A1A02085584}

\begin{document}

\maketitle

 \theoremstyle{plain}
\newtheorem{Thm}{Theorem}[section]
\newtheorem{Lemma}[Thm]{Lemma}
\newtheorem{Cor}[Thm]{Corollary}
\newtheorem{Prop}[Thm]{Proposition}
\newtheorem{Claim}[Thm]{Claim}

 \theoremstyle{definition}
\newtheorem{Def}[Thm]{Definition}
\newtheorem{Conj}[Thm]{Conjecture}
\newtheorem{Rem}[Thm]{Remark}
\newtheorem{Prb}[Thm]{Problem}
\newtheorem{Ex}[Thm]{Example}
\newtheorem{Fact}[Thm]{Fact}

\begin{abstract}
    Mekler developed a way to produce a pure group from any given structure where the construction preserves $\kappa$-stability for any cardinal $\kappa$. Not only the stability, it is known that his construction preserves various model-theoretic properties such as simplicity, NIP, and NTP$_2$. Inspired by the last result, we show that the construction also preserves NTP$_1$(NSOP$_2$) and NSOP$_1$. As a corollary, we obtain that if there is a theory of finite language which is non-simple NSOP$_1$, or which is NSOP$_2$ but has SOP$_1$, then there is a pure group theory with the same properties, respectively.
\end{abstract}
\vspace{10pt}
\vspace{20pt}

\section{Introduction}
Suppose a structure $M$ of finite language has more than one element, then there is a graph $N$ which is bi-interpretable with $M$ \cite[Thm 5.5.1]{Hodges}. 
This implies if $M$ has a model-theoretic property like stability and simplicity, then one can find a graph which has the same properties of $M$.
Unlike the graph, it is not easy to see whether there is a group which preserves a model-theoretic property of $M$.
A partial answer to this problem was found by Mekler. In \cite{Mekler}, he constructed a group $G$ where Th$(G)$ has the same stability spectrum as Th$(M)$. 
This group is not bi-interpretable with $M$, however, so it does not preserve all the properties of $M$. For example, even though $M$ is $\aleph_0$-categorical, the group $G$ may not be.
\par
Later, it is proved that many other properties related with Shelah's classification program are preserved by Mekler's construction.
Baudisch and Pentzel proved that simplicity is preserved by the construction, and assuming stability, Baudisch proved that CM-triviality is also preserved \cite{Baudisch}. 
Recently, Chernikov and Hempel proved that the construction preserves NIP, $k$-dependence, and NTP$_2$ \cite{Chernikov Hempel}. 
Thus, it is natural to expect that the construction preserves NSOP$_1$ and other non-simple theories \cite[Conjecture 1]{Chernikov Hempel}.
\par
In this paper we show that the conjecture is true for the following tree properties; NTP$_1$(NSOP$_2$) and NSOP$_1$. 
To prove them, we analogously follow the argument used in the proof of preservation of NTP$_2$ in \cite{Chernikov Hempel}. 
The difference is that parameters witnessing TP$_2$ formula is an array, not a tree. Hence, we find an appropriate generalized indiscernibility for each properties substituting the role of mutual indiscernibility.
We use one of the tree indiscernibility, called strong indiscernibility (see Definition \ref{i indxd indisc}) \cite{KKS, Takeuchi Tsuboi}.
\par
SOP$_1$ formula also have parameters of a tree, but the strong indiscernibility is not much helpful. 
We recall \cite[Proposition 2.4]{Kaplan Ramsey}, the equivalent conditions of NSOP$_1$, to obtain parameters of array $\omega\times2$ called SOP$_1$-array (see Definition \ref{comb indiscernible}).
\par
We preview the corollaries of the main results.
\begin{Cor}
 \begin{enumerate}
     \item There is a non-simple NSOP$_1$ pure group theory.
     \item If there is an NSOP$_2$ theory which has SOP$_1$, then there is a pure group theory with the same properties.
 \end{enumerate}
\end{Cor}
The first one is obtain by the preservation of simplicity and NSOP$_1$. 
Any example of non-simple NSOP$_1$ theory on finite language can be transformed into a pure group by the construction. 
Similarly, any NSOP$_2$ theory with SOP$_1$ on finite language can be transformed into a pure group.
Thus, NSOP$_1$ and NSOP$_2$ are equivalent if and only if they are equivalent on the pure group theories.
\par
The paper is organized as follows.
In section 2, we introduce the notions about strong indiscerniblility on trees from \cite{KKS} and \cite{Takeuchi Tsuboi}. 
In section 3, using strong indiscernibility, we find equivalent conditions of NTP$_1$. And then, we define an SOP$_1$-array and find equivalent conditions of NSOP$_1$. 
In section 4, we describe and summarize definitions and facts of Mekler's construction following by \cite{Hodges} and \cite{Chernikov Hempel}. 
In section 5, we first observe some combinatorial remarks on trees in \cite{Dzamonja Shelah}, then show our main results that Mekler's construction preserves NTP$_1$ and NSOP$_1$.

\section{Tree indiscernibility}

Consider a tree $^{<\lambda}\kappa$ of height $\lambda$ which has $\kappa$ many branches. Each element in the tree can be regarded as a string. We denote $\langle\rangle$ as an empty string, $0^\alpha$ as a string of $\alpha$ many zeros, and $\alpha$ as a string $\langle\alpha\rangle$ of length one.

\begin{Def} \label{tree lang}
Fix a tree $^{<\lambda}\kappa$, and let $\eta, \nu, \xi \in \null^{<\lambda}\kappa$.\par
\begin{enumerate}[(1)]
 \item (Ordering) $\eta \lhd \nu$ if $\nu \lceil \alpha = \eta$ for some ordinal $\alpha \in$ dom($\nu)$.
 \item (Meet) $\xi = \eta \wedge \nu$ if $\xi$ is the meet of $\eta$ and $\nu$, i.e., $\xi=\eta \lceil \beta$, when $\beta = \bigcup \{ \alpha \leq $ dom$(\eta) \cap $dom$(\nu) \hspace{1mm}| \hspace{1mm} \eta \lceil \alpha=\nu \lceil \alpha \}$. For $\bar{\eta}\in\null^{<\lambda}\kappa$, $\bar{\nu}$ is the meet closure of $\bar{\eta}$ if $\bar{\nu}=\{ \eta_1 \wedge \eta_2 | \eta_1, \eta_2 \in \bar{\eta} \}$.
 \item (Incomparability) $\eta \bot \nu$ if they are $\unlhd$-incomparable, i.e., $\neg (\eta\unlhd\nu)$ and $\neg(\nu\unlhd\eta)$.
 \item (Lexicographic order) $\eta<_{lex}\nu$ if
  \begin{itemize}
   \item[(a)] $\eta \lhd \nu$, or
   \item[(b)] $\eta \bot \nu$ and $\exists \alpha (\eta \lceil \alpha=\nu \lceil \alpha \hspace{1mm}$and$\hspace{1mm} \eta(\alpha)<\nu(\alpha))$.
  \end{itemize}
\end{enumerate}
\end{Def}

\begin{Def} \label{str lang}
A \emph{strong language} $L_0$ is defined by the collection $\{ \lhd, \wedge, <_{lex} \}$ \par
We may view the tree $^{<\lambda}\kappa$ as an $L_0$-structure.
\end{Def}

\begin{Def}
 A tree $B=\null^{<\lambda'}\kappa'$ is called a \emph{subtree} of $A=\null^{<\lambda}\kappa$ if $B\subseteq A$ and the inclusion map is an embedding in the language $\{\unlhd, <_{lex}\}$.
\end{Def}

 Fix a complete first order theory $T$(with language $L$). Let $\mathcal{M} \models T$ be a monster model. From now on, we will work in this $\mathcal{M}$.
 
\begin{Def} \cite{KKS} \label{i indxd indisc}
Fix a structure $\mathcal{I}$ with language $L_{\mathcal{I}}$. For a set $\{b_i | i \in \mathcal{I}\}$, we say it is \emph{$\mathcal{I}$-indexed indiscernible} if for any finite $\bar{i}$ and $\bar{j}$ from $\mathcal{I}$,
 \begin{center}
  qftp($\bar{i})_\mathcal{I}$ = qftp($\bar{j})_\mathcal{I}$ $\Rightarrow$ $(b_i)_{i\in \bar{i}} \equiv (b_j)_{j\in \bar{j}}$.
 \end{center}
 \quad $\mathcal{I}$ is called the index structure. In particular, we say a set $\{b_\eta | \eta \in \null^{<\lambda}\kappa\}$ is \emph{strongly indiscernible} if it is $\mathcal{I}$-indexed indiscernible for $\mathcal{I}$ the $L_0$-structure on $^{<\lambda}\kappa$.
\end{Def}

\begin{Rem} \label{str indisc prop}
Let $\{ a_\eta | \eta \in \null^{<\lambda} \kappa \}$ be a strongly indiscernible tree.
 \begin{enumerate}[(1)]
     \item For all $\nu_1 , \nu_2 \in \null^\lambda \kappa$, $(a_{\nu_1 \lceil \alpha })_{\alpha < \lambda} \equiv (a_{\nu_2 \lceil \alpha })_{\alpha < \lambda}$ 
     \item For all $\bar{\eta_1}, \bar{\eta_2} \in \null^{<\lambda} \kappa$, if $\bar{\nu_i}$ is the meet-closure of $\bar{\eta_i}$ for each $i=1,2$, then qftp($\bar{\eta_1}$) = qftp($\bar{\eta_2}$) implies qftp($\bar{\nu_1}$) = qftp($\bar{\nu_2}$)
     \item For all $\eta \bot \nu \in \null^{<\lambda}\kappa$, and $\xi\in\null^{<\lambda}\kappa$, $\eta<_{lex}\nu \Rightarrow a_\eta a_\nu \equiv a_{\xi^\frown 0} a_{\xi^\frown 0}$
     \item For any $\eta \in \null^{<\lambda} \kappa$, the tree $(a_{0^\frown \eta})_{\eta \in \null^{<\lambda} \kappa} $ is strongly indiscernible over $(a_{\nu_1 \lceil \alpha })_{\alpha \in dom(\eta)}$
 \end{enumerate}
\end{Rem}
\begin{proof}
See \cite{Chernikov Ramsey} and \cite{KKS}.
\end{proof}

\begin{Def}  \cite{KKS} \label{EM and basedness}
Let $\mathcal{I}$ be an index structure.
 \begin{enumerate}[(1)]
     \item The \emph{EM-type} of a set of parameters $A=\{ a_i \,|\, i \in \mathcal{I} \}$, EM$_{\mathcal{I}}(A)$, is the collection of formulas $\varphi(x_{i_1}, \dots , x_{i_n})$ in $L$ with variables $\{ x_i \,|\, i \in \mathcal{I} \}$ such that for all $j_1, \dots , j_n \in \mathcal{I}$, if $j_1 \dots j_n\equiv^{qf}i_1 \dots i_n$, then $\models\varphi(a_{j_1}, \dots , a_{j_n})$.
    \item A set $B=\{ b_\eta \,|\, \eta \in \mathcal{I} \}$ is \emph{based on} a set $A=\{ a_\nu \,|\, \nu \in \mathcal{I} \}$  if for all $\eta_1, \dots , \eta_n \in \mathcal{I}$ and for all $\varphi(x_{\eta_1}, \dots , x_{\eta_n})$ in $L$, there exists some $\nu_1, \cdots , \nu_n \in \mathcal{I}$ such that \par
     \begin{itemize}
         \item[(a)] $\nu_1 \dots \nu_n\equiv^{qf}_\mathcal{I} \eta_1 \dots \eta_n$, and
         \item[(b)] $b_{\eta_1} \dots b_{\eta_n} \equiv_\varphi a_{\nu_1} \dots a_{\nu_n}$
     \end{itemize}
     \quad In particular, when $\mathcal{I}$ is $L_0$-structure $\null^{<\lambda}\kappa$, we say $B$ is \emph{strongly based on} $A$ whenever $B$ is based on $A$.
 \end{enumerate}
\end{Def}

\begin{Rem} \label{EM eq basedness}
Let $B=\{ b_\eta \,|\, \eta \in \mathcal{I} \}$ and $A=\{ a_\nu \,|\, \nu \in \mathcal{I} \}$. Then $B$ is based on $A$ if and only if $B\models$EM$_{\mathcal{I}}(A)$.
\end{Rem}

\begin{Def} \label{modeling property}
For an index structure $\mathcal{I}$, we say $\mathcal{I}$-indexed indiscernibles have the \emph{modeling property} if given any $A=\{ a_\nu | \nu \in \mathcal{I} \}$, there is an $\mathcal{I}$-indexed indiscernible $B=\{ b_\eta | \eta \in \mathcal{I} \}$ such that $B$ is based on $A$(or equivalently, $B\models$EM$_{\mathcal{I}}(A)$).
\end{Def}

\begin{Fact} \cite{Takeuchi Tsuboi} \label{str indisc have MP}
Let $^{<\omega}\omega$ be the universe of the index structure. The strong indiscernibles have the modeling property.
\end{Fact}
As a corollary, we may assume the index structure as $^{<\lambda}\kappa$ where both $\lambda$ and $\kappa$ are infinite cardinals.

\section{Tree properties}

\begin{Def} \label{k dist sib}
We say a subset $\{ \eta_i \,|\, i<k \} \subseteq \null^{<\lambda}\kappa$ is a collection of \emph{$k$ distant siblings} if given $i_1\neq i_2$ and $j_1\neq j_2$, all of which are less than $k$, $\eta_{i_1}\wedge\eta_{i_2}=\eta_{j_1}\wedge\eta_{j_2}$.
\end{Def}

We may extend the result in Proposition \ref{str indisc prop}(3).
\begin{Rem} \label{dist sib eq to 0 to k-1}
Let $\{ a_\eta \,|\, \eta \in \null^{<\lambda} \kappa \}$ be a strongly indiscernible tree. For any collection of distant siblings $\{ \eta_i \,|\, i<k \} \subseteq \null^{<\lambda}\kappa$ and for any $\xi\in\null^{<\lambda}\kappa$, if $\eta_i<_{lex}\eta_j$ for each $i<j<k$ then $a_{\eta_0}\cdots a_{\eta_{k-1}} \equiv a_{\xi^\frown 0}\cdots a_{\xi^\frown k-1}$
\end{Rem}

\begin{Def} \cite{Dzamonja Shelah, KKS, Takeuchi Tsuboi} \label{tree properties}
Fix $k\geq2$.
 \begin{enumerate}[(1)]
     \item $\varphi(x;y)$ has SOP$_1$ if there is a $(a_\eta \,|\, \eta \in \null^{<\omega}2 )$ such that 
      \begin{itemize}
          \item[(a)] For all $\eta\in\null^{\omega}2$, $\{ \varphi(x;a_{\eta \lceil \alpha}) \,|\, \alpha<\omega \}$ is consistent,
          \item[(b)] For all $\xi, \nu \in \null^{<\omega}2$, if $\xi^\frown0\unlhd\nu$, then $\{ \varphi(x;a_{\xi^\frown1 }), \varphi(x;a_{\nu}) \}$ is inconsistent.
      \end{itemize}
     \item $\varphi(x;y)$ has SOP$_2$ if there is a $(a_\eta \,|\, \eta \in \null^{<\omega}2 )$ such that 
      \begin{itemize}
          \item[(a)] For all $\eta\in\null^{\omega}2$, $\{ \varphi(x;a_{\eta \lceil \alpha}) \,|\, \alpha<\omega \}$ is consistent,
          \item[(b)] For all $\xi, \nu \in \null^{<\omega}2$, if $\xi\perp\nu$, then $\{ \varphi(x;a_{\xi}), \varphi(x;a_{\nu}) \}$ is inconsistent.
      \end{itemize}
     \item $\varphi(x;y)$ has the tree property of the first kind (TP$_1$) if there is $(a_\eta \,|\, \eta \in \null^{<\omega} \omega )$ such that 
      \begin{itemize}
          \item[(a)] For all $\eta\in\null^{\omega} \omega$, $\{ \varphi(x;a_{\nu \lceil \alpha}) | \alpha<\omega \}$ is consistent,
          \item[(b)] For all $\eta \bot \nu \in \null^{<\omega}\omega$, $\{ \varphi(x;a_{\eta }), \varphi(x;a_{\nu}) \}$ is inconsistent.
      \end{itemize}
     \item $\varphi(x;y)$ has weak $k$-TP$_1$ if there is $(a_\eta \,|\, \eta \in \null^{<\omega} \omega )$ such that 
      \begin{itemize}
          \item[(a)] For all $\eta\in\null^{\omega} \omega$, $\{ \varphi(x;a_{\nu \lceil \alpha}) | \alpha<\omega \}$ is consistent,
          \item[(b)] For any collection of distant siblings $\{ \eta_i \,|\, i<k \}$, $\{ \varphi(x;a_{\eta_i}) \,|\, i<k \}$ is inconsistent.
      \end{itemize}
     \item We say $T$ has TP$_1$ (resp. SOP$_1$, SOP$_2$) if there is a formula having TP$_1$ (resp. SOP$_1$, SOP$_2$). If not, we say $T$ is NTP$_1$ (resp. NSOP$_1$, NSOP$_2$). We say $T$ has weak-TP$_1$ if there is a formula having $k$-TP$_1$ for some $k$.
 \end{enumerate}
\end{Def}

\begin{Fact} \cite{Chernikov Ramsey} \label{weak ktp1 eq to tp1} 
\begin{enumerate}
    \item $\varphi(x;y)$ has TP$_1$ if and only if there is a strongly indiscernible tree $(a_\eta\,|\,\eta\in \null^{<\omega}\omega)$ such that \par
    \begin{enumerate} [(a)]
        \item $\{\varphi(x,a_{\eta  \lceil \alpha})\,|\, \alpha<\omega\}$ is consistent for some $\eta\in ^\omega \omega$.
        \item $\{\varphi(x,a_{\nu^\frown i})\,|\,i<\omega\}$ is pairwise inconsistent for some $\nu\in\null^{<\omega}\omega$.
    \end{enumerate}
        \item $T$ has weak-TP$_1$ if and only if $T$ has TP$_1$.
\end{enumerate}
\end{Fact}


\begin{Rem}
$\varphi(x;y)$ has weak-TP$_1$ if and only if there is a strongly indiscernible tree $(a_\eta|\eta\in\null^{<\omega} \omega)$ such that \par
    \begin{enumerate} [(a)]
        \item $\{\varphi(x,a_{0^i})|i<\omega\}$ is consistent.
        \item $\{\varphi(x,a_{\nu^\frown i})|i<\omega\}$ is $k$-inconsistent for some $\nu\in\null^{<\omega}\omega$ and $k\geq2$.
    \end{enumerate}
\end{Rem}

To prove the main theorem, we first establish the characterization of given model theoretic property. For example, in \cite{Chernikov Hempel}, Chernikov and Hempel stated the following proposition cited from \cite{Chernikov ntp2};
\begin{Fact} \label{ntp2 eq}
Let $T$ be a theory and $\mathcal{M}\models T$ a monster model. Let $\kappa:=|T|^+$. The following are equivalent:
 \begin{enumerate}
     \item T is NTP$_2$.
     \item for any array $(a_{i,j} : i\in\kappa, j\in\omega)$ of finite tuples with mutually indiscernible rows and a finite tuple $b$, there is some $\alpha\in\kappa$ satisfying the following: \par 
     for any $i>\alpha$, there is some $b'$ such that 
     \begin{enumerate}
         \item $(a_{i,j}\,|\,j<\omega)$ is indiscernible over $b'$, and
         \item tp$(b/a_{i,0})$=tp$(b'/a_{i,0})$.
     \end{enumerate}
 \end{enumerate}
\end{Fact}

The basic idea of \ref{ntp2 eq} is to find an indiscernible sequence over some $b'$ which has the same type with given $b$ over some element in that sequence. The trick in the proof can be seen on \cite[Lemma 1.4]{Shelah simple}, too. We apply the same trick on NTP$_1$ with the notions of strong indiscernibility.
 
\begin{Prop} \label{ntp1 main cor}
Let $\kappa>2^{|T|}$ be some sufficiently large regular cardinal. Then TFAE.
 \begin{enumerate}
     \item $T$ is NTP$_1$
     \item For any strongly indiscernible tree $(a_\eta \,|\, \eta \in \null^{<\kappa} \kappa )$ of finite tuples and a finite tuple $b$, there is some $\beta<\kappa$ and $b'$ such that
      \begin{enumerate}[(a)]
         \item $(a_{{0^\beta}^\frown i} \,|\, i<\omega)$ is indiscernible over $b'$, and
         \item tp$(b/a_{{0^\beta}^\frown 0})$ = tp$(b'/a_{{0^\beta}^\frown 0})$,
      \end{enumerate}
     \item For any strongly indiscernible tree $(a_\eta \,|\, \eta \in \null^{<\kappa} \kappa )$ of finite tuples and a finite tuple $b$, there is some $\gamma<\kappa$ satisfying the following:\par
     for any $\beta>\gamma$, there is some $b'$ such that
     \begin{enumerate}[(a)]
         \item $(a_{{0^\beta}^\frown i} \; | \;i<\omega)$ is indiscernible over $b'$, and
         \item tp$(b/a_{{0^\beta}^\frown 0})$ = tp$(b'/a_{{0^\beta}^\frown 0})$.
      \end{enumerate}
 \end{enumerate}
\end{Prop}
\begin{proof}
(1) $\Rightarrow$ (2). Assume $T$ is NTP$_1$, and let $A = (a_\eta \,|\, \eta \in \null^{<\kappa} \kappa )$ and $b$ be given.\par
By pigeonhole principle, there is a subsequence $(\alpha_i \,|\, i<\omega)$ in the set of successor ordinals smaller than $\kappa$ such that for all $i<j<\omega$, ${\alpha_i}^+<\alpha_j<\kappa$ and tp$(a_{0^{\alpha_i}}/b)$ = tp$(a_{0^{\alpha_j}}/b)$. We inductively define a subtree $(a'_\eta \,|\, \eta \in \null^{<\omega} \omega )$ in $A$ as follows;
 \begin{itemize}
     \item $a'_{\langle\rangle}=a_{0^{\alpha_0}}$,
     \item for any $\eta\in  \null^{<\omega} \omega$, if $a'_\eta=a_\rho$ for some $\rho\in \null^{<\kappa} \kappa$ with Dom$(\rho)=\alpha_i$, then for each $j<\omega$, $a'_{\eta^\frown j}=a_{{\rho'}^\frown j}$ where $\rho'=\rho^\frown0^{\delta}$ for some $\delta$ so that Dom$({\rho'}^\frown j)=\alpha_{i+1}$.
 \end{itemize}
Note that $( a'_{0^i} \,|\, i<\omega )$ is a subsequence of $(a_{0^\alpha} \,|\, \alpha<\kappa)$, and for any $\eta, \nu\in \null^{<\omega} \omega$, if $a'_{\eta}=a_{\eta'}$, $a'_{\nu}=a_{\nu'}$, $a'_{\eta\wedge\nu}=a_{\xi}$ for $\eta', \nu'$, and $\xi \in\null^{<\kappa} \kappa$, then $\xi\neq\eta'\wedge\nu'$. Moreover, the subtree is strongly indiscernible, too.
 \par
Let $p(x,a'_0)$ := tp$(b/a'_0)$ and $q(x)$ := $\bigcup_{i<\omega}p(x,a'_i)$. We claim that $q$ is consistent.
\par
 Suppose not. Then by compactness and strong indiscernibility, there is a formula $\varphi(x,y)$ such that $\varphi(x,a'_0) \in p(x,a'_0)$ and $\{ \varphi(x,a'_i) \,|\, i<\omega \}$ is $k$-inconsistent for some $k<\omega$. 
 On the other hand, $b$ realizes $\bigcup_{i<\omega}p(x,a'_{0^i})$, so it realizes $\{ \varphi(x,a'_{0^i}) \,|\, i<\omega \}$.
 As a result, $\varphi$ has weak $k$-TP$_1$, and fact \ref{weak ktp1 eq to tp1} further says that $T$ has TP$_1$. \par
By claim, we can find a realization $b' \models q(x)$. \par
We may assume $( a'_{i} \,|\, i<\omega )$ is indiscernible over $b'$ by Ramsey and compactness. Note that tp$(b/a'_0)$ = tp$(b'/a'_0)$ since $b' \models p(x,a'_0)$. Thus, if $a'_i=a_{{0^\beta}^\frown i}$ for some $\beta$, then this $\beta$ and $b'$ is the desired one.
\par
(2) $\Rightarrow$ (3). Assume (2) holds. For a strongly indiscernible tree $(a_\eta \,|\, \eta \in \null^{<\kappa} \kappa )$ and a finite tuple $b$, we will say $Q(\beta)$ holds on $(a_\eta \,|\, \eta \in \null^{<\kappa} \kappa )$ and $b$ when there is a $b'$ such that $(a_{{0^\beta}^\frown i} \; | \;i<\omega)$ is indiscernible over $b'$, and tp$(b/a_{{0^\beta}^\frown 0})$ = tp$(b'/a_{{0^\beta}^\frown 0})$.
Suppose there is a strongly indiscernible tree $(a_\eta \,|\, \eta \in \null^{<\kappa} \kappa )$ and a finite tuple $b$ such that for any $\gamma<\kappa$, there is a $\beta>\gamma$ which does not satisfy $Q$. Then, since cf$(\kappa)=\kappa$, we can find a cofinal map $f:\kappa\rightarrow\kappa$ such that for any $i,j<\kappa$, $1<f(i)^+<f(j)$, and $f(i)$ does not satisfy $Q$.\par
Now, construct the following map $g:\null^{<\kappa} \kappa \rightarrow \null^{<\kappa} \kappa$ . First, if $\eta=\langle\rangle$, then $g(\eta)=\langle\rangle$ For other non-empty string $\eta\in\null^{<\kappa} \kappa$, Dom$(g(\eta))=$ Sup$\{f(l)^+\,|\,l\in\text{Dom}(\eta)\}$, and for each $i<\text{Dom}(g(\eta))$,
\begin{center}
$(g(\eta))(i)=
\begin{cases}
\eta(j) & \text{if } i=f(j)^+ \text{ for some } j<\kappa, \\
0 & \text{otherwise}.
\end{cases}$
\end{center}
Then we define $a'_\eta=a_{g(\eta)}$. The subtree $(a'_\eta\,|\,\eta\in\null^{<\kappa} \kappa)$ is strongly indiscernible, and for each $\beta<\kappa$, $Q(\beta)$ does not hold on $(a'_\eta\,|\,\eta\in\null^{<\kappa} \kappa)$ and $b$. This contradicts to (2).
\par
(3) $\Rightarrow$ (1). Assume (3) holds. Suppose $T$ has TP$_1$ witnessed by $\varphi(x;y)$ and $(a_\eta \,|\, \eta \in \null^{<\kappa} \kappa )$. 
We may assume the tree $(a_\eta \,|\, \eta \in \null^{<\kappa} \kappa )$ is strongly indiscernible by the modeling property. 
Note (a) and (b) in \ref{tree properties}(3) still hold. Take $b\models\{ \varphi(x,a_{0^\alpha}) \,|\, \alpha<\kappa \}$. By the assumption, we have some ordinals $\gamma<\beta<\kappa$ and $b'$ such that
 \begin{enumerate}[(a)]
    \item $(a_{{0^\beta}^\frown i} \; | \;i<\omega)$ is indiscernible over $b'$, and
    \item tp$(b/a_{{0^\beta}^\frown 0})$ = tp$(b'/a_{{0^\beta}^\frown 0})$.
 \end{enumerate}
\par
Choose an automorphism $\sigma$ where it fixes $a_{{0^\beta}^\frown 0}$ and sends $b'$ to $b$. Denote $(a'_i \,|\, i<\omega)$ to $(\sigma(a_{{0^\beta}^\frown i})\,|\, i<\omega)$, then it is indiscernible over $b$. Since $a'_0$ is as same as $a_{{0^\beta}^\frown 0}$, 
$b\models\varphi(x,a'_0)$, and then the indiscernibility implies that for all $i<\omega$, $b\models\varphi(x,a'_i)$. But this is a contradiction because $\{\varphi(x,a'_i) | i<\omega\}$ is 2-inconsistent, as well as $\{\varphi(x,a_{{0^\beta}^\frown i}) \,|\, i<\omega\}$ is.

\end{proof}

Analogously, We find a lemma for SOP$_1$. Unlike the case of TP$_1$, we cannot use the strong indiscernibility. 
For instance, let $A$ be a tree which has the inconsistency condition of SOP$_1$ and let $B$ be a strong indiscernible tree based on $A$. 
There is no guarantee that $B$ has the inconsistency condition. 
Hence, we need another indiscernibility which matches up to SOP$_1$.\par
To find this, we recall results in \cite{Chernikov Ramsey} and \cite{Kaplan Ramsey}.

\begin{Fact} \cite{Chernikov Ramsey} \label{comb in tree}
Suppose $\varphi(x;y)$ with the tree $(c_\eta\,|\,\eta\in\null^{<\kappa}2)$ have SOP$_1$ where $\kappa\geq 2^{|T|}$. Then there is a sequence $(\eta_i,\nu_i)_{i<\omega}$ of elements of $\null^{<\kappa}2$ such that
 \begin{enumerate}
     \item $c_{\eta_i}\equiv_{c_{\eta_{<i}}c_{\nu_{<i}}} c_{\nu_i}$ for all $i<\omega$,
     \item $\{\varphi(x;c_{\eta_i})\,|\, i<\omega \}$ is consistent, and
     \item $\{\varphi(x,c_{\nu_i})\,|\, i<\omega \}$ is 2-inconsistent.
 \end{enumerate}
\end{Fact}

Kaplan and Ramsey \cite{Kaplan Ramsey} proved more general result about SOP$_1$.
\begin{Fact} \label{sop1 and comb str}
The following are equivalent;
 \begin{enumerate}
     \item $T$ has SOP$_1$.
     \item There is a formula $\varphi$ and an array $(c_{i,j})_{i<\omega, j<2}$ so that
      \begin{enumerate}
          \item $c_{i.0}\equiv_{c_{<i,0}c_{<i,1}} c_{i,1}$ for all $i<\omega$
          \item $\{\varphi(x;c_{i,0})\,|\, i<\omega \}$ is consistent
          \item $\{\varphi(x,c_{i,1})\,|\, i<\omega \}$ is 2-inconsistent.
      \end{enumerate}
     \item There is a formula $\varphi$ and an array $(c_{i,j})_{i<\omega, j<2}$ so that
      \begin{enumerate}
          \item $c_{i.0}\equiv_{c_{<i,0}c_{<i,1}} c_{i,1}$ for all $i<\omega$
          \item $\{\varphi(x;c_{i,0})\,|\, i<\omega \}$ is consistent
          \item $\{\varphi(x,c_{i,1})\,|\, i<\omega \}$ is $k$-inconsistent for some $k\geq2$.
      \end{enumerate}
 \end{enumerate}
\end{Fact}

From these facts, we derive a new kind of indiscernibility.
\begin{Def} \label{comb indiscernible}
We say $(c_{i,j})_{i<\omega, j<2}$ is an \emph{SOP$_1$-array indiscernible over $A$} if
 \begin{enumerate}
     \item $(c_{i,0}c_{i,1})_{i<\omega}$ is an (order) indiscernible sequence over $A$, and
     \item $c_{i.0}\equiv_{Ac_{<i,0}c_{<i,1}} c_{i,1}$ for all $i<\omega$.
 \end{enumerate}
\end{Def}
Note that we may replace the array in \ref{sop1 and comb str} to an indiscernible SOP$_1$-array .

\begin{Rem}
Let an indiscernible SOP$_1$-array $(c_{i,j})_{i<\omega, j<2}$ be given.
 \begin{enumerate}
     \item For any $\kappa>\omega$, there is a indiscernible SOP$_1$-array $(c'_{i,j})_{i<\kappa, j<2}$ such that $c_{i,j}=c'_{i,j}$ for all $i<\omega$ and $j<2$.
     \item  For any $n<\omega$, $(c_{i,j})_{n\leq i<\omega, j<2}$ is an SOP$_1$-array indiscernible over $\{ c_{i,j}\,|\,i<n, j<2 \}$.
 \end{enumerate}
\end{Rem}

\begin{Prop} \label{nsop1 main cor}
Let $\kappa>2^{|T|}$ be some sufficiently large regular cardinal. Then TFAE.
 \begin{enumerate}[(1)]
     \item $T$ is NSOP$_1$ 
     \item For any indiscernible SOP$_1$-array $(a_{i,j})_{i<\kappa, j<2}$ of finite tuples and a finite tuple $b$, there is some $\beta<\kappa$ and some $b'$ such that
      \begin{enumerate}[(a)]
         \item tp$(b,a_{\beta,0})$ = tp$(b',a_{0,1})$,
         \item $(a_{i,1}\; | \;i<\omega)$ is indiscernible over $b'$.
      \end{enumerate}
     \item For any indiscernible SOP$_1$-array $(a_{i,j})_{i<\kappa, j<2}$ of finite tuples and a finite tuple $b$, there is some $\gamma<\kappa$ satisfying the following:\par
     for any number $\beta>\gamma$, there is some $b'$ such that
      \begin{enumerate}[(a)]
         \item tp$(b,a_{\beta,0})$ = tp$(b',a_{0,1})$,
         \item $(a_{i,1}\; | \;i<\omega)$ is indiscernible over $b'$.
      \end{enumerate}
 \end{enumerate}
\end{Prop}
\begin{proof}
(1) $\Rightarrow$ (2). Assume $T$ is NSOP$_1$, and let $(a_{i,j})_{i<\kappa, j<2}$ and $b$ in (2) be given. 
Using the pigeonhole principle, take an indiscernible SOP$_1$-subarray $(a'_{i,0}a'_{i,1})_{i<\omega}$ in $(a_{i,0}a_{i,1})_{i<\kappa}$ where tp$(a'_{i,0}/b)$=tp$(a'_{j,0}/b)$ for all $i, j<\omega$.\par
Now let $p(x,y)=$ tp$(b,a'_{0,0})$, and $q(x)=\bigcup_{i<\omega}p(x,a'_{i,1})$. We claim that $q$ is consistent.
Suppose not. Then by indiscernibility and compactness, there is a formula $\varphi(x,y)$ in $p(x,y)$ such that $\{\varphi(x,a'_{i,1})\,|\,i<\omega\}$ is $k$-inconsistent for some natural number $k$. 
On the other hand, since $\varphi(x,a'_{i,0})$ is in tp$(b/a'_{i,0})$ for each $i<\omega$, $\{\varphi(x,a'_{i,0})\,|\,i<\omega\}$ is consistent.
Hence $\varphi$ has SOP$_1$ by Fact \ref{sop1 and comb str}, which is a contradiction.\par
Let $\bar{y}=(y_i)_{i<\omega}$ where for each $i<\omega$, $|y_i|=|a'_{i,1}|$, and let $\Pi(x,\bar{y})$ be the union of $\bigcup_{i<\omega}p(x,y_i)$, $\Phi(\bar{y})$, and $\Psi(x,\bar{y})$ where $\Phi(\bar{y})=$tp$((a'_{i,1})_{i<\omega})$ and $\Psi(x, \bar{y})$ means that $(y_i)_{i<\omega}$ is indiscernible over $x$. 
By Ramsey and compactness, $\Pi(x,\bar{y})$ is consistent. Let $(b'',(a''_{i,1})_{i<\omega})$ be a realization of $\Pi(x,\bar{y})$.
Since $\Phi(\bar{y})=$tp$((a'_{i,1})_{i<\omega})$=tp$((a_{i,1})_{i<\omega})$, we have $(a_{i,1}\; | \;i<\omega)$ is indiscernible over some $b'$ where tp$(b',a_{0,1})$ = tp$(b'',a''_{0,1})=p(x,y_0)$ = tp$(b,a'_{0,0})$. 
Note $a'_{0,0}=a_{\beta,0}$ for some $\beta<\kappa$.\par

(2) $\Rightarrow$ (3). For an indiscernible SOP$_1$-array $(a_{i,j})_{i<\kappa, j<2}$ and a finite tuple $b$, we will say $Q(\beta)$ holds on $(a_{i,j})_{i<\kappa, j<2}$ and $b$ when there is a $b'$ such that tp$(b/a_{\beta,0})$ = tp$(b'/a_{0,1})$, and $(a_{i,1}\; | \;i<\omega)$ is indiscernible over $b'$.
Assume (2) holds but (3) does not hold, that is, there is an indiscernible SOP$_1$-array $(a_{i,j})_{i<\kappa, j<2}$ and a tuple $b$ such that for any $\gamma<\kappa$, there is some $\beta>\gamma$ which does not satisfy $Q$.
From this and cf$(\kappa)=\kappa$, we choose an increasing sequence $(\beta_i)_{i<\kappa}$ of ordinal numbers such that for each $\beta_i$, $Q$ does not hold.\par
Now take a subarray $(a_{\beta_i,0}a_{\beta_i,1})_{i<\kappa}$ in $(a_{i,0}a_{i,1})_{i<\kappa}$. 
This array is still indiscernible, so by (2), there is some $j<\kappa$ and some $b'$ such that tp$(b,a_{\beta_j,0})$ = tp$(b',a_{\beta_0,1})$, and $(a_{\beta_i,1}\; | \;i<\omega)$ is indiscernible over $b'$.
Since tp$((a_{\beta_i,1})_{i<\omega})$=tp$((a_{i,1})_{i<\omega})$ by indiscernibility, we may assume that there is some $j<\kappa$ and some $b'$ such that tp$(b,a_{\beta_j,0})$ = tp$(b',a_{0,1})$, and $(a_{i,1}\; | \;i<\omega)$ is indiscernible over $b'$. This contradicts that $Q(\beta_j)$ does not hold.
\par
(3) $\Rightarrow$ (1). Assume (3) holds. Suppose $T$ has SOP$_1$. 
By fact\ref{sop1 and comb str}(2) and compactness, we have a formula $\varphi(x;y)$ which witnesses SOP$_1$ with an indiscernible SOP$_1$-array $(a_{i,j})_{i<\kappa, j<2}$.
Let $b$ be a realization of $\bigwedge_{i<\kappa}\varphi(x,a_{i,0})$. By assumption, there is some ordinals $\gamma<\beta<\kappa$ and some $b'$ such that
      \begin{enumerate}[(a)]
         \item tp$(b,a_{\beta,0})$ = tp$(b',a_{0,1})$,
         \item $(a_{i,1}\; | \;i<\omega)$ is indiscernible over $b'$.
      \end{enumerate}
From (a), we have $\models\varphi(b',a_{0,1})$, and then from (b), we have $\models\varphi(b',a_{1,1})$. This contradicts that $\{\varphi(x,a_{i,1})\,|\, i<\kappa \}$ is 2-inconsistent.
\end{proof}

\section{Mekler's construction}
We recall the definitions and facts from \cite{Hodges}.

For a graph $A$ and its vertices $a$ and $b$, we say $R(a,b)$ if $a$ and $b$ are connected by a single edge in $A$.
\begin{Def} \label{def of nice graph}
 A graph $A$ which has at least two vertices is called \emph{nice} if
 \begin{enumerate} [(a)]
    \item For any two distinct vertices $a$ and $b$, there is some vertex $c\neq a,b$ such that $R(a,c)$ but $\neg R(b,c)$;
    \item There are no triangles nor squares.
 \end{enumerate}
\end{Def}

\begin{Def} \cite{Baudisch, Mekler}  \label{def of Mekler gp}
 Fix an odd prime $p$. For a nice graph $A$, let $F(A)$ be the free nilpotent group of class 2 and exponent $p$ generated freely by the vertices of $A$. Assume that $A$ is enumerated with some relation $<$ not in the original language. Then the \emph{Mekler group of $A$}, denoted by $G(A)$, is defined as follows;
 \begin{center}
     $G(A)=F(A)/\langle\{[a,b]\,|\,a,b\in A, a<b, \text{and } A\models R(a,b)\}\rangle$.
 \end{center}
\end{Def}
\par 
In other words, $G(A)$ is a group defined in the variety of nilpotent groups of class 2 and exponent $p$ such that the generators are the vertices of $A$ and that for any $a$ and $b$ in $G(A)$, $[a,b]=1$ if and only if $a<b$ and $A\models R(a,b)$.
\par
\medskip
We can see the definition in the point of view of vector spaces. Let $Z(F(A))$ be the center of $F(A)$. Then both $Z(F(A))$ and $F(A)/Z(F(A))$ are all elementary abelian $p$-groups so they can be considered as a $\mathbb{F}_p$-vector space with basis $\{[a,b]\,|\,a,b\in A, a<b\}$ and $\{a/Z(F(A))\,|\,a\in A\}$ respectively. 
The same is true for the Mekler group $G(A)$. If $Z(G(A))$ is the center of $G(A)$, then both $Z(G(A))$ and $G(A)/Z(G(A))$ can be considered as a $\mathbb{F}_p$-vector space. The basis of $Z(G(A))$ is $\{[a,b]\,|\,a,b\in A, a<b, \text{ and } A\models \neg R(a,b)\}$, and the basis of $G(A)/Z(G(A))$ is $\{a/Z(F(A))\,|\,a\in A\}$.

\begin{Def}
 For any element $g, h$ of $G(A)$, we say
 \begin{enumerate}
     \item $g\sim h$ if $C(g)=C(h)$, where $C(g)$ is the centraliser of $g$ in $G(A)$,
     \item $g\approx h$ if for some $c$ in the center $Z(G)$ and some $r$ $(0\leq r<p)$, $h=g^r \cdot c$,
     \item $g\equiv_Z h$ if $g\cdot Z(G)=h\cdot Z(G)$.
 \end{enumerate}
\end{Def}

\begin{Rem}
 For any element $g, h$ of $G(A)$, $g\equiv_Z h \Rightarrow g\approx h \Rightarrow g\sim h$.
\end{Rem}

\begin{Def}
 Let $g$ be an element in $G(A)$.
 \begin{enumerate}
     \item $g$ is \emph{isolated} if every non-central element of $G(A)$ which commutes with $g$ is $\approx$-equivalent to $g$.
     \item We say an element $g$ is \emph{of type $q$} if $q$ is the number of $\approx$-classes in the $\sim$-class of $g$.
     \item We say $g$ is \emph{of type $q^\iota$} (resp. $q^\nu$) if $g$ is of type $q$ and isolated (resp. of type $q$ and not isolated).
 \end{enumerate}
\end{Def}

\begin{Def}
 For every element $g$ of type $p$, we say an element $b$ is a \emph{handle of $g$} if it is of type $1^\nu$ and commutes with $g$.
\end{Def}

 As a remark, we note that for any $g$ of type $p$, the handle of $g$ exists and is unique up to $\sim$-equivalence.

\begin{Fact} \label{classification of elements of G(A)}
Every non-central element of $G(A)$ is of exactly one of the four type $1^\nu, 1^\iota, p-1, p$; the classes of elements of each type are 0-definable.
\end{Fact}

Now, let $G$ be a model of Th$(G(A))$. Let us say an element of $G$ is proper if it is not a product of any elements of type $1^\nu$ in $G$.

\begin{Def}
\begin{enumerate}
    \item A \emph{$1^\nu$-transversal of $G$}, denoted by $X^\nu$, is a set consisting of one representative for each $\sim$-class of elements of type $1^\nu$ in $G$.
    \item A \emph{$p$-transversal of $G$}, denoted by $X^p$, is a set of pairwise $\sim$-inequivalent proper elements of type $p$ in $G$ which is maximal with the property that if $Y$ is a finite subset of $X^p$ and all elements of $Y$ have the same handle, then $Y$ is a independent modulo the subgroup generated by all elements of type $1^\nu$ in $G$ and $Z(G)$.
    \item A \emph{$1^\iota$-transversal of $G$}, denoted by $X^\iota$, is a set of representatives of $\sim$-classes of proper elements of type $1^\iota$ in $G$ shich is maximal independent modulo the subgroup generated by all elements of types $1^\iota$ and $p$ in $G$, together with $Z(G)$.
    \item A subset $X$ of $G$ is called a \emph{transversal of $G$} if it is the union of some $1^\nu$-transversal $X^\nu$, A $p$-transversal $X^p$, and $1^\iota$-transversal $X^\iota$ of $G$.
\end{enumerate}
\end{Def}

Note that all the sets in the above definition are definable.

\begin{Fact} \label{interpretation in Mekler gp}
Let $A$ be a nice graph. For a model $G\models$ Th$(G(A))$, define an interpretation $\Gamma$ such that $\Gamma(G)$ is a graph where the set of vertices is $\{g\in G\,|\, g$ is a noncentral element of type $1^\nu\}/\sim$ and the edge relation is $\{([g]_{\sim},[h]_{\sim})\,|\,[g,h]=1 \text{ in }G\}$. Then  $\Gamma(G)\models$ Th$(A)$.
\end{Fact}
From \ref{interpretation in Mekler gp}, we see that if $X^\nu$ is a $1^\nu$-transversal,  then the set can be regarded as a graph which models Th$(A)$.

\begin{Fact} \label{factorization of Mekler gp}
Let $C$ be an infinite nice graph, and $G\models $ Th$(G(C))$. 
If $X=X^\nu\cup X^p\cup X^\iota$ is a transversal of $G$, then there is a subgroup $H_X\leq Z(G)$ such that $G=\langle X \rangle \times H_X$ for some $H_X\leq Z(G)$. Moreover, if $G$ is saturated and uncountable, then both $\Gamma(G)$ and $H_X$ are also saturated.
\end{Fact}
Since $H_X$ is an elementary abelian $p$-group, we sometimes say $G$ is isomorphic to $\langle X \rangle \times \langle H_X\rangle$.

\begin{Fact}
Let $G$ is a saturated model of Th$(G(C))$ and let $\kappa=|G|$. If $X=X^\nu\cup X^p\cup X^\iota$ is a transversal of $G$, then
\begin{enumerate}[(a)]
    \item for any $x^\nu\in X^\nu$, the cardinality of $\{x^p\in X^p\,|\,x^\nu \text{ is the handle of } x^p\}$ is either zero or $\kappa$, and
    \item $|X^\iota|=\kappa$
\end{enumerate}
\end{Fact}

As we mentioned in \ref{interpretation in Mekler gp}, $X^\nu$ can be regarded as a graph where two vertices are joined (connected by a single edge) if they commute in $G$. 
In this point of view, we can find a supergraph by extending the set of vertices to $X$ and then giving the edge relation with the same rule. 
Then each $x^p\in X^p$ is joined to a unique vertex in $X^\nu$, which is the handle of $x^p$, while each $x^\iota\in X^\iota$ is joined to no vertex.
\par
This kind of supergraph is called a cover. See \cite{Chernikov Hempel} for more precise proof.
\par
\medskip
We give more facts from \cite{Chernikov Hempel}.
\begin{Fact} \label{ext in MG}
 \begin{enumerate}
    \item Let $G$ be a saturated model of Th$(G(C))$, and let $X$ and $H_X$ be the sets in \ref{factorization of Mekler gp} so that $G=\langle X \rangle \times H_X$. If $f$ is a bijection between two small sets $Y=Y^\nu\cup Y^p \cup Y^\iota$ and $Z=Z^\nu\cup Z^p\cup Z^\iota$ of $X$ with the following properties;
    \begin{enumerate}
        \item $f(Y^\nu)=Z^\nu$, $f(Y^p)=Z^p$, and $f(Y^\iota)=Z^\iota$,
        \item for any $y^p\in Y^p$, the handle of $y^p$ is same as the handle of $f(y^p)$,
        \item tp$(Y^\nu)$=tp$(Z^\nu)$ in $\Gamma$,
    \end{enumerate}
    then $f$ can be extended to an automorphism $\sigma$ of $G$.
    \par
    Moreover, for any $\bar{h}$, $\bar{k}\in H_X$, if tp$(\bar{h})$=tp$(\bar{k})$ in $H_X$, then we may assume $\sigma$ sends $\bar{h}$ to $\bar{k}$.
    \item Let $G$ be a model of Th$(G(C))$, and $\bar{x}=\bar{x}^\nu{}^\frown\bar{x}^p{}^\frown\bar{x}^\iota$ and $\bar{y}$ be two small tuples of variables. Then there is a partial type $\pi(\bar{x},\bar{y})$ such that $G\models \pi(\bar{a},\bar{b})$ if and only if we can extend $\bar{a}$ to a transversal $X$ of $G$ and find $H$ containing $\bar{b}$ so that $H$ is an independent set in $Z(G)$ and $G=\langle X \rangle \times \langle H\rangle$.
 \end{enumerate}
\end{Fact}

\section{Preseervation}

In \cite{Dzamonja Shelah}, džamonja and Shelah introduced a way to choose a monochromatic subtree when the given tree $\null^{<\kappa}2$ is colored by $\theta$ many colors where $|\theta|<\kappa$ and $\kappa$ is a regular cardinal. The following is the key observation.

\begin{Fact} \label{monochrom fact}
Let $\kappa$ be a regular cardinal and let $\theta$ be a set of colors such that $|\theta|<\kappa$. For any coloring $f:\null^{<\kappa}2\rightarrow\theta$, there is a color $c\in\theta$ and an element $\nu^\ast\in\null^{<\kappa}2$ such that for any $\nu\in\null^{<\kappa}2$ satisfying $\nu^\ast\unlhd\nu$, there is $\rho\in\null^{<\kappa}2$ with $\nu\unlhd\rho$, $f(\rho)=c$.
\end{Fact}

Let us say a subtree $B\subseteq A$ is $f$-monochromatic if there is a color $c$ such that for all $b\in B$, $f(b)=c$.

\begin{Lemma} \label{monochrom lemma}
Let $\kappa$ be an uncountable regular cardinal and $f:\null^{<\kappa}2\rightarrow\omega$ be a coloring.
 \begin{enumerate}
     \item If $\varphi(x;y)$ and $(a_\eta \,|\, \eta \in \null^{<\kappa} 2 )$ witness SOP$_2$, then there is a $f$-monochromatic subtree $(a'_\eta \,|\, \eta \in \null^{<\omega} 2 )$ of $(a_\eta \,|\, \eta \in \null^{<\kappa} 2 )$ such that $\varphi$ and $(a'_\eta \,|\, \eta \in \null^{<\omega} 2 )$ witness SOP$_2$.
     \item If $\varphi(x;y)$ and $(a_\eta \,|\, \eta \in \null^{<\kappa} 2 )$ witness SOP$_1$, then there is a $f$-monochromatic subtree $(a'_\eta \,|\, \eta \in \null^{<\omega} 2 )$ of $(a_\eta \,|\, \eta \in \null^{<\kappa} 2 )$ such that $\varphi$ and $(a'_\eta \,|\, \eta \in \null^{<\omega} 2 )$ witness SOP$_1$.
     \item If $\varphi(x;y)$ and $(a_\eta \,|\, \eta \in \null^{<\kappa} \kappa )$ witness TP$_1$, then there is a $f$-monochromatic subtree $(a'_\eta \,|\, \eta \in \null^{<\omega} \omega )$ of $(a_\eta \,|\, \eta \in \null^{<\kappa} \kappa )$ such that $\varphi$ and $(a'_\eta \,|\, \eta \in \null^{<\omega} \omega )$ witness TP$_1$.
 \end{enumerate}
\end{Lemma}
\begin{proof}
(1) Let $c$ and $\nu^\ast$ be the element in Fact \ref{monochrom fact}. We construct a subtree $(a'_\eta \,|\, \eta \in \null^{<\omega} 2 )$ in $(a_\eta | \eta \in \null^{<\kappa} 2 )$ inductively as follows;
\begin{itemize}
    \item $a'_{\langle\rangle}= a_\rho$ where $\rho$ is any element satisfying $\nu^\ast\unlhd\rho$ and $f(\rho)=c$,
    \item for any $\eta\in\null^{<\omega} 2$, if $a'_\eta$ = $a_\rho$ for some $\rho\in\null^{<\kappa} 2$ , then $a'_{\eta^\frown0}= a_{\rho'}$ where $\rho^\frown0\unlhd\rho'$ and $f(\rho')=c$, and $a'_{\eta^\frown1}= a_{\rho'}$ where $\rho^\frown1\unlhd\rho'$ and $f(\rho')=c$.
\end{itemize}
$\varphi$ and the subtree witness SOP$_2$.\par
(2) We use a different version of the Fact \ref{monochrom fact} which has one more condition on $\rho$: $\rho$ is of the form $\xi^\frown1$ for some $\xi$. 
This can be proved by the same way used in \ref{monochrom fact}, so we omit the proof.\par
Define $a'_{\langle\rangle}$ as same as the previous one, but give a small modification in the induction step. For any $\eta\in\null^{<\omega} 2$, let $a'_\eta$ = $a_\rho$ for some $\rho\in\null^{<\kappa} 2$. 
Take $a'_{\eta^\frown1}= a_{\rho'}$ where $\rho\unlhd\rho'$, $f(\rho')=c$ and $\rho'$ is of the form $\xi^\frown1$ for some $\xi$.
Then take $a'_{\eta^\frown0}= a_{\rho''}$ where $\xi^\frown0\unlhd\rho''$ and $f(\rho'')=c$. $\varphi$ and the subtree witness SOP$_1$, too.\par
(3) In $(a_\eta \,|\, \eta \in \null^{<\kappa} \kappa )$, consider the subtree $(a_\eta \,|\, \eta \in \null^{<\kappa} 2 )$. 
Find $(a'_\eta \,|\, \eta \in \null^{<\omega} 2 )$ as in the proof of (1), and then construct $(a''_\eta \,|\, \eta \in \null^{<\omega} \omega )$ inductively as follows;
\begin{itemize}
    \item $a''_{\langle\rangle}= a'_{\langle\rangle}$,
    \item for any $\eta\in\null^{<\omega}\omega$, if $a''_\eta=a_\rho$ for some $\rho\in\null^{<\omega}2$, then $a''_{\eta^\frown i}=a'_{\rho^\frown {1^i}^\frown 0 }$ for each $i<\omega$.
\end{itemize}
$\varphi$ and the last subtree witness TP$_1$, too.
\end{proof}

\begin{Thm} \label{ntp1 preservation}
For any infinite nice graph C, Th(G(C)) is NTP$_1$ if and only if Th(C) is NTP$_1$.
\end{Thm}
\begin{proof}
Since $C$ is interpretable in $G(C)$, if Th$(C)$ has TP$_1$, then Th$(G(C))$ also has TP$_1$.\par
Suppose Th$(C)$ is NTP$_1$ and Th$(G(C))$ has TP$_1$. Let $G$ be a monster model of Th$(G(C))$, $X$ be a transversal of $G$ so that $G=\langle X \rangle \times \langle H\rangle$ for some $H$ as in Fact \ref{factorization of Mekler gp}. 
We have a formula $\varphi(x,y)$ and a tree $(a_\eta \,|\, \eta \in \null^{<\kappa} \kappa )$ of finite tuples in $G$ for some sufficiently large regular cardinals $\kappa$ so that they witness TP$_1$. 
Note that for each $\eta \in \null^{<\kappa} \kappa$, $a_\eta$ is of the form $t_\eta(\bar{x}_\eta, \bar{h}_\eta)$ for some terms $t_\eta\in L_G$, and for some small tuples $\bar{x}_\eta =\bar{x}^\nu_\eta{}^\frown\bar{x}^p_\eta{}^\frown\bar{x}^\iota_\eta \in X$, and $\bar{h}_\eta\in H$.
\par
By Lemma \ref{monochrom lemma}, we may assume $t_\eta = t\in L_G$, and $|\bar{x}^\nu_\eta|, |\bar{x}^p_\eta|, |\bar{x}^\iota_\eta|, |\bar{h}_\eta|$ are constant for all $\eta \in \null^{<\omega} \omega$. 
To obtain handle correspondence, add handles of elements in the tuple $\bar{x}^p_\eta$ to the beginning of $\bar{x}^\nu_\eta$ for all $\eta \in \null^{<\omega} \omega$.
\par
Taking $\varphi'(x,y'):=\varphi(x,t(y'))$ with $|y'|=|\bar{x}_\eta^\frown\bar{h}_\eta|$ and $\bar{b}_\eta := \bar{x}_\eta^\frown\bar{h}_\eta$, we have $\varphi'\in L_G$ and the tree $(\bar{b}_\eta \,|\, \eta \in \null^{<\omega} \omega )$ still satisfy TP$_1$.
\par
By modeling property of strong indiscernibility and compactness, we can find $(\bar{c}_\eta \,|\, \eta \in \null^{<\kappa} \kappa )$ with $\bar{c}_\eta = \bar{y}_\eta^\frown\bar{m}_\eta = \bar{y}^\nu_\eta{}^\frown\bar{y}^p_\eta{}^\frown\bar{y}^\iota_\eta{}^\frown\bar{m}_\eta$ to be a strongly indiscernible tree where $(\bar{c}_\eta \,|\, \eta \in \null^{<\omega} \omega )$ based on $(\bar{b}_\eta | \eta \in \null^{<\omega} \omega )$. 
Note $\varphi'$ and the tree still witness TP$_1$. Also, by \ref{ext in MG}(2), we can assume each $\bar{y}_\eta$ and $\bar{m}_\eta$ are in some $Y$ and $M$ where $Y$ is a transversal of $G$ and $M$ is an independent set in $Z(G)$ and $G=\langle Y \rangle \times \langle M\rangle$.
Let $c$ be a realization of $\bigwedge_{\alpha<\kappa} \varphi'(x,\bar{c}_{0^\alpha})$.
Write $c=s(y,m)$ for some terms $s\in L_G$, and for some tuples $y =y^\nu{}^\frown y^p{}^\frown y^\iota \in Y$, and $m\in M$. 
Again, To obtain handle correspondence, add handles of elements in the tuple $y^p$ to the beginning of $y^\nu$.
Take $\psi(x',y')=\varphi'(s(x'),y')$, then for all $\eta\bot\nu\in\null^{<\kappa}\kappa$, $\{\psi(x',\bar{c}_\eta),\psi(x',\bar{c}_\nu)\}$ is inconsistent and $y^\frown m$ realizes $\bigwedge_{\alpha<\kappa} \psi(x',\bar{c}_{0^\alpha})$. Since $y^\frown m\cap\bigcup\{\bar{c}_{0^\alpha}\,|\,\alpha<\kappa\}$ is finite, we may assume the tree is strongly indiscernible over $y^\frown m\cap\bigcup\{\bar{c}_{0^\alpha}\,|\,\alpha<\kappa\}$.
\par
Now, consider $y^\nu$ and the tree $(\bar{y}_\eta^\nu \,|\, \eta \in \null^{<\kappa} \kappa )$ in $Y^\nu$.
Applying \ref{interpretation in Mekler gp}, we can regard the elements as vertices of a graph. This graph satisfies NTP$_1$ theory Th$(C)$, so there is some $\gamma$ satisfying \ref{ntp1 main cor}(2).
Then for each $\beta^+>\gamma$, we have a tuple ${y'}^\nu$ such that tp$_\Gamma(y^\nu/\bar{y}^\nu{}_{{0^\beta}^\frown 0})$=tp$_\Gamma({y'}^\nu/\bar{y}^\nu{}_{{0^\beta}^\frown 0})$ and tp$_\Gamma(\bar{y}^\nu{}_{{0^\beta}^\frown0 }/{y'}^\nu)$=tp$_\Gamma(\bar{y}^\nu{}_{{0^\beta}^\frown 1}/{y'}^\nu)$.
\par
On the other hand, observe that Th$(\langle M\rangle)$ is a theory of vector spaces, so that the theory is stable and has quantifier elimination. 
Then for tuple $m$ and the tree $(\bar{m}_\eta \,|\, \eta \in \null^{<\omega} \omega )$ in $\langle M\rangle$, we can apply Corollary \ref{ntp1 main cor} to have some $\gamma'$ satisfying \ref{ntp1 main cor}(2).
\par
Fix a successor ordinal $\beta^+$ larger then $\gamma$ and $\gamma'$, and let ${y'}^\nu$ and $m'$ be the tuples given by Corollary \ref{ntp1 main cor}. Recall that the tree $(\bar{y}_\eta^\frown\bar{m}_\eta \, |\, \eta \in \null^{<\omega} \omega )$ is strongly indiscernible over $y^\frown m\cap\bigcup\{\bar{c}_{0^\alpha}\,|\,\alpha<\omega\}$. So, as in \cite[Theorem 5.6]{Chernikov Hempel}, we can find a handle preserving bijection which can be extended by Fact \ref{ext in MG}(1) to have
 \begin{enumerate}
    \item tp$_G(ym/\bar{y}_{{0^\beta}^\frown 0}\bar{m}_{{0^\beta}^\frown 0})$=tp$_G(y'm'/\bar{y}_{{0^\beta}^\frown 0}\bar{m}_{{0^\beta}^\frown 0})$, and
    \item tp$_G(\bar{y}_{{0^\beta}^\frown0 }\bar{m}_{{0^\beta}^\frown0 }/y'm')$=tp$_G(\bar{y}_{{0^\beta}^\frown 1}\bar{m}_{{0^\beta}^\frown 1}/y'm')$.
 \end{enumerate}
From these conditions, we have $G\models\psi(y'm',\bar{y}_{{0^\beta}^\frown 0}\bar{m}_{{0^\beta}^\frown 0})\wedge\psi(y'm',\bar{y}_{{0^\beta}^\frown 1}\bar{m}_{{0^\beta}^\frown 1})$, but this contradicts that for any $\eta\bot\nu\in\null^{<\kappa} \kappa$,  $\{\psi(x',\bar{y}_\eta\bar{m}_\eta),\psi(x',\bar{y}_\nu\bar{m}_\nu)\}$ is inconsistent.

\end{proof}

\begin{Thm}
For any infinite nice graph C, Th(G(C)) is NSOP$_1$ if and only if Th(C) is NSOP$_1$.
\end{Thm}
\begin{proof}
Since $C$ is interpretable in $G(C)$, if Th$(C)$ has SOP$_1$, then Th$(G(C))$ also has SOP$_1$.
\par
Suppose Th$(C)$ is NSOP$_1$ and Th$(G(C))$ has SOP$_1$. Again, we take $G$ to be a monster model of Th$(G(C))$, $X$ to be a transversal of $G$ so that $G=\langle X \rangle \times \langle H\rangle$ for some $H$ as in Fact \ref{factorization of Mekler gp}. 
We have a formula $\varphi(x,y)$ and a tree $(a_\eta | \eta \in \null^{<\kappa} 2 )$ of finite tuples in $G$ for some sufficiently large regular cardinals $\kappa$. 
Note that for each $\eta \in \null^{<\kappa} 2$, $a_\eta$ is of the form $t_\eta(\bar{x}_\eta, \bar{h}_\eta)$ for some terms $t_\eta\in L_G$, and for some small tuples $\bar{x}_\eta =\bar{x}^\nu_\eta{}^\frown\bar{x}^p_\eta{}^\frown\bar{x}^\iota_\eta \in X$, and $\bar{h}_\eta\in H$.
\par
By Lemma \ref{monochrom lemma}, we may assume $t_\eta = t\in L_G$, and $|\bar{x}^\nu_\eta|, |\bar{x}^p_\eta|, |\bar{x}^\iota_\eta|, |\bar{h}_\eta|$ are constant for all $\eta \in \null^{<\kappa} 2$.
To obtain handle correspondence, add handles of elements in the tuple $\bar{x}^p_\eta$ to the beginning of $\bar{x}^\nu_\eta$ for all $\eta \in \null^{<\kappa} 2$.
\par
Taking $\varphi'(x,y'):=\varphi(x,t(y'))$ with $|y'|=|\bar{x}_\eta^\frown\bar{h}_\eta|$ and $\bar{a}'_\eta := \bar{x}_\eta^\frown\bar{h}_\eta$, we have a formula $\varphi'\in L_G$ and the tree $(\bar{a}'_\eta | \eta \in \null^{<\kappa} 2 )$ also witness SOP$_1$.
\par
By Fact \ref{comb in tree}, we can choose an array $(\bar{b}_{i,j})_{i<\omega, j<2}$ in $(\bar{a}'_\eta \,|\, \eta \in \null^{<\omega} \omega )$ such that
 \begin{enumerate}
     \item $\bar{b}_{i,0}\equiv_{\bar{b}_{<i,0}\bar{b}_{<i,1}} \bar{b}_{i,1}$ for all $i<\omega$,
     \item $\{\varphi'(x;\bar{b}_{i,0})\,|\, i<\omega \}$ is consistent, and
     \item $\{\varphi'(x,\bar{b}_{i,1})\,|\, i<\omega \}$ is 2-inconsistent.
 \end{enumerate}
Consider an indiscernible sequence $(\bar{c}_{i,0}\bar{c}_{i,1})_{i<\kappa}$ realizing  EM$((\bar{b}_{i,0}\bar{b}_{i,1})_{i<\omega})$. 
The new array $(\bar{c}_{i,j})_{i<\kappa, j<2}$ is an indiscernible SOP$_1$-array and satisfies aforementioned conditions.
Also, $\bar{c}_{i,j}$ is of the form $\bar{y}_{i,j}^\frown\bar{m}_{i,j}$ and there are some $Y$ and $M$ such that $Y$ is a transversal of $G$ extended from $\{\bar{y}_{i,j}\,|\,i<\kappa, j<2\}$, $M$ is an independent set in $Z(G)$ containing $\{\bar{m}_{i,j}\,|\,i<\kappa, j<2\}$, and $G=\langle Y\rangle\times\langle M\rangle$. 
\par
Let $c$ be a realization of $\bigwedge_{i<\kappa} \varphi'(x,\bar{c}_{i,0})$.
Write $c=s(y,m)$ for some terms $s\in L_G$, and for some tuples $y =y^\nu{}^\frown y^p{}^\frown y^\iota \in Y$, and $m\in M$. 
Again, To obtain handle correspondence, add handles of elements in the tuple $y^p$ to the beginning of $y^\nu$.
Take $\psi(x',y')=\varphi'(s(x'),y')$, then for all $i<j<\kappa$, $\{\psi(x',\bar{c}_{i,1}),\psi(x',\bar{c}_{j,1})\}$ is inconsistent and $y^\frown m$ realizes $\bigwedge_{i<\kappa} \psi(x',\bar{c}_{i,0})$. 
Since $y^\frown m\cap\bigcup\{\bar{c}_{i,0}\,|\,i<\kappa\}$ is finite, we may assume the array is an SOP$_1$-array indiscernible over $y^\frown m\cap\bigcup\{\bar{c}_{i,0}\,|\,i<\kappa\}$.
\par
Now, consider $y^\nu$ and the array $(\bar{y}_{i,j}^\nu | i<\kappa, j<2 )$ in $Y^\nu$.
Applying \ref{interpretation in Mekler gp}, we can regard the elements as vertices of a graph. This graph is a model of NSOP$_1$ theory Th$(C)$, so there is some $\gamma$ satisfying \ref{nsop1 main cor}(2).
Then for each $\beta>\gamma$, we have a tuple ${y'}^\nu$ such that tp$_\Gamma(y^\nu/{\bar{y}^\nu}_{i,0})$=tp$_\Gamma({y'}^\nu/{\bar{y}^\nu}_{0,1})$ and tp$_\Gamma({\bar{y}^\nu}_{0,1}/{y'}^\nu)$=tp$_\Gamma({\bar{y}^\nu}_{1,1}/{y'}^\nu)$.
\par
On the other hand, observe that Th$(\langle M\rangle)$ is a theory of vector spaces, so that the theory is stable and has quantifier elimination. 
Then for the tuple $m$ and the array $(\bar{m}_{i,j}^\nu | i<\omega, j<2 )$ in $\langle M\rangle$, we can apply Corollary \ref{nsop1 main cor} to have some $\gamma'$ satisfying \ref{nsop1 main cor}(2).
\par
Fix some ordinal $\beta$ larger then $\gamma$ and $\gamma'$, and let ${y'}^\nu$ and $m'$ be the tuples given by \ref{nsop1 main cor}. We will find a handle preserving bijection which can be extended by Fact \ref{ext in MG}(1) to have
 \begin{enumerate}
    \item tp$_G(ym/\bar{y}_{\beta,0}\bar{m}_{\beta,0})$=tp$_G(y'm'/\bar{y}_{0,1}\bar{m}_{0,1})$, and
    \item tp$_G(\bar{y}_{0,1}\bar{m}_{0,1}/y'm')$=tp$_G(\bar{y}_{1,1}\bar{m}_{1,1}/y'm')$,
 \end{enumerate}
for some tuple $y'={y'}^\nu {y'}^p {y'}^\iota$.
\par
 From the above two type equivalence, we have $G\models\psi(y'm',\bar{y}_{0,1}\bar{m}_{0,1})\wedge\psi(y'm',\bar{y}_{1,1}\bar{m}_{1,1})$, and this contradicts that for any $i<j<\kappa$,  $\{\psi(x',\bar{y}_{i,1}\bar{m}_{i,1}),\psi(x',\bar{y}_{j,1}\bar{m}_{j,1})\}$ is inconsistent.
\par
We finish our proof by illustrating the way to find tuple $y'={y'}^\nu {y'}^p {y'}^\iota$. For a finite tuple $\bar{z}$ and a natural number $k$, denote $(\bar{z})_k$ to be the $k$-th element of $\bar{z}$.\par
Let ${y}^\iota$ be the tuple of length $l$. For each $i<l$, if $(y^\iota)_i$ is $({\bar{y}^\iota}_{\beta,0})_k$ for some $k$, then choose $y'^\iota_i$ to be $({\bar{y}^\iota}_{0,1})_k$, and if $(y^\iota)_i$ is not in ${\bar{y}^\iota}_{\beta,0}$, then choose $(y'^\iota)_i$ to be any element in $Y^\iota\setminus {\bar{y}^\iota}_{0,1}$.
\par
Let ${y}^p$ be the tuple of length $l'$. For each $i<l'$, either $(y^p)_i$ is in ${\bar{y}^p}_{\beta,0}$ or not. 
Assume first that $(y^p)_i$ is $({\bar{y}^p}_{\beta,0})_k$ for some $k$. Then choose $(y'^p)_i$ to be $({\bar{y}^p}_{0,1})_k$. 
Second, suppose $(y^p)_i$ is not in ${\bar{y}^p}_{\beta,0}$. Recall that $(y^\nu)_i$ is the handle of $(y^p)_i$. We have either $(y^\nu)_i$ is in ${\bar{y}^\nu}_{\beta,0}$ or not.
If $(y^\nu)_i$ is $({\bar{y}^\nu}_{\beta,0})_k$, then choose $(y'^p)_i$ to be any element in $Y^p\setminus {\bar{y}^p}_{0,1}$ whose handle is $({\bar{y}^\nu}_{0,1})_k$.
Otherwise, choose $(y^\nu)_i$ to be any element in $Y^p$ whose handle is $(y'^\nu)_i$.
\par
Mapping each element in $y{\bar{y}}_{\beta,0}$ to ${y'}{\bar{y}}_{0,1}$ by their natural order, we have a bijection which satisfies hypotheses in \ref{ext in MG}(1).
As a result, tp$_G(y/\bar{y}_{\beta,0})$=tp$_G(y'/\bar{y}_{0,1})$.
\par
It remains to check the bijection from $\bar{y}_{0,1}y'$ to $\bar{y}_{1,1}y'$ is well-defined. Suppose the $(\bar{y}_{0,1})_k$ is equal to the $(y')_{k'}$ for some $k$ and $k'$. 
Since tp$_G(y/\bar{y}_{\beta,0})$=tp$_G(y'/\bar{y}_{0,1})$, $(\bar{y}_{\beta,0})_k=(y)_{k'}$.
This element is in $y^\frown m\cap\bigcup\{\bar{c}_{i,0}\,|\,i<\kappa\}$, hence for every $i<\kappa$, $(\bar{y}_{i,0})_k$ is $(y)_{k'}$. 
By the definition of indiscernible SOP$_1$-array, the $(\bar{y}_{i,1})_k$ is $(y)_{k'}$, too.
Thus $(y')_{k'}$, $(\bar{y}_{0,1})_k$, $(\bar{y}_{1,1})_k$, and $(y)_{k'}$ are all the same elements. 
This comes out again if we assume $k$-th element of $\bar{y}_{1,1}$ is equal to the $k'$-th element of $y'$. 
Therefore, we have a well-defined bijection from $\bar{y}_{0,1}y'$ to $\bar{y}_{1,1}y'$.

\end{proof}

\begin{Cor}
 \begin{enumerate}
     \item There is a non-simple NSOP$_1$ pure group theory.
     \item If there is an NSOP$_2$ theory which has SOP$_1$, then there is a pure group theory with the same properties.
 \end{enumerate}
\end{Cor}
\begin{proof}
(1) Fix a structure $M$ of finite language such that Th$(M)$ is non-simple and NSOP$_1$. 
Then by \cite[Theorem 5.5.1, Exercise 5.5.9]{Hodges}, there is a nice graph $C$ bi-interpretable with $M$. 
Since Mekler's construction preserves simplicity and NSOP$_1$, the theory of Mekler group of $C$ is non-simple and NSOP$_1$, too.
\par
(2) Follow the same argument above.
\end{proof}

\end{document}